\documentclass[preprint,12pt]{elsarticle}
\usepackage{amssymb}
\usepackage{float}
\usepackage{amsmath}
\usepackage{amssymb}
\usepackage{amsthm}
\usepackage{mathrsfs}
\usepackage[mathscr]{euscript}
\usepackage[ruled,vlined]{algorithm2e}
\usepackage{arydshln}
\usepackage{multirow}
\usepackage{multicol}

\usepackage{lineno,hyperref}
\usepackage{xcolor}
\hypersetup{
    colorlinks,
    linkcolor={blue},
    citecolor={blue},
    urlcolor={blue}
}

\journal{Acta Astronautica}

\newtheorem{theorem}{Theorem}[section]

\newtheorem{lemma}[theorem]{Lemma}
\newtheorem{assumption}[theorem]{Assumption}
\newtheorem{proposition}[theorem]{Proposition}
\theoremstyle{remark}
\newtheorem{remark}[theorem]{Remark}
\theoremstyle{definition}
\newtheorem{definition}[theorem]{Definition}

\begin{document}

\begin{frontmatter}

\title{A New Shape-Based Multiple-Impulse Strategy for Coplanar Orbital Maneuvers}

\author{Amir Shakouri}
\address{Department of Aerospace Engineering, Sharif University of Technology, Tehran 14588-89694, Iran}
\author{Maryam Kiani\fnref{label01}}
\fntext[label01]{Corresponding Author, Email Addresses: \href{mailto:a_shakouri@ae.sharif.edu}{a$\_$shakouri@ae.sharif.edu} (A. Shakouri), \href{mailto:kiani@sharif.edu}{kiani@sharif.edu} (M. Kiani), \href{mailto:pourtak@sharif.edu}{pourtak@sharif.edu} (S.H. Pourtakdoust).}
\author{Seid H. Pourtakdoust}
\address{Center for Research and Development in Space Science and Technology, Sharif University of Technology, Tehran 14588-89694, Iran}

\begin{abstract}

A new shape-based geometric method (SBGM) is proposed for generation of multi-impulse transfer trajectories between arbitrary coplanar oblique orbits via a heuristic algorithm. The key advantage of the proposed SBGM includes a significant reduction in the number of design variables for an $N$-impulse orbital maneuver leading to a lower computational effort and energy requirement. The SBGM generates a smooth transfer trajectory by joining a number of confocal elliptic arcs such that the intersections share common tangent directions. It is proven that the well-known classic Hohmann transfer and its bi-elliptic counterpart between circular orbits are special cases of the proposed SBGM. The performance and efficiency of the proposed approach is evaluated via computer simulations whose results are compared with those of optimal Lambert maneuver and traditional methods. The results demonstrate a good compatibility and superiority of the proposed SBGM in terms of required energy effort and computational efficiency. 

\end{abstract}


\end{frontmatter}

%\linenumbers

\section{Introduction}
\label{S:1}

Multiple-impulse orbital maneuvers have been widely utilized in many space missions. The most well-known multiple impulsive maneuver for coplanar orbits was first proposed by Walter Hohmann \citep{01} in year 1925. Then, Lion et al. \citep{02} have utilized the primer vector for optimal impulsive time-fixed approaches. The primer vector has been implemented for both circular and elliptical orbits \citep{03,04}. In addition, the effect of path and thrust constraints on impulsive maneuvers have been investigated \citep{05,06}, and the optimal impulsive transfer in presence of time constraints is driven \citep{07}. Further, various heuristic optimization algorithms such as the genetic algorithm (GA) \citep{08}, the particle swarm optimization (PSO) \citep{09,09p}, and the simulated annealing \citep{09pp} have been proposed to design an optimal impulsive trajectory (also the reader can refer to \citep{09ppp} and references therein for more details). However, the Lambert’s approach has traditionally been utilized for conic trajectories between any two spatial points in space within a specified time interval. Due to vast applicability of the Lambert method \citep{10}, multiple-revolution \citep{11,12,13}, perturbed \citep{14,15} and optimized Lambert solutions \citep{16} have been subsequently developed for the multiple impulsive maneuvers. Relative Lambert solutions are also developed \citep{17} and utilized for impulsive rendezvous of spacecraft \citep{18}. Moreover, the effect of orbital perturbations including the Earth oblateness \citep{19,20} and the third body \citep{21} on the impulsive trajectories has been investigated to make more efficient transfers. The impulsive orbital transfers have also been addressed against the low-thrust trajectories \citep{22}, and it is shown that the impulsive transfer can be viewed as the limit of the low-thrust transfer when its number of revolutions increases. 

Impulse vectors can be applied perpendicular to the orbit plane for plane change maneuvers where in \citep{22p}, analytic fuel optimal solutions are provided. Impulsive maneuvers sometimes are applied tangentially at the origin and/or destination points. Altman et al. \citep{23} have proposed the hodograph theory to generate tangent maneuvers, which can improve the rendezvous safety \citep{24}. Zhang et al. \citep{25} have analytically analyzed the existence of tangent solutions between elliptical orbits, followed by further study on utility of two-impulse tangent orbital maneuvers \citep{26,27}. In addition, the reachable domain of spacecraft with just a single tangent impulse is studied in \citep{28,29}. 

Shape-based approaches for orbital maneuvers are based on the trajectory geometry. In continuous propulsion systems, some candidate spiral functions can be considered to define the spacecraft trajectory \citep{30,31,32}. Utilizing these functions makes the resultant optimization problem have fewer parameters in comparison to the main problem. The parameters are selected such that satisfy the problem constraints. An efficient function for the trajectory shaping is obviously a smooth curve that passes through the initial and final positions and also avoids large control efforts. 

The current paper presents a novel shape-based geometric method (SBGM) to form a smooth multiple-impulse coplanar transfer trajectory between any two arbitrary spatial points in space. The smoothness property is guaranteed via solving a system of nonlinear equations. Smoothness constraint reduces the number of required design variables for the resultant optimization problem indeed, and lead to a near optimal solution in terms of the energy requirement. Unlike the continuous propulsion case, the proposed approach in the current work provides a final smooth trajectory via joining a number of piecewise elliptic arcs, while observing the continuity requirements. The proposed SBGM reduces the number of design variables and enhances the ability of seeking for global optimum in a bounded region. The applicability and efficiency of the proposed SBGM are verified through comparison with the traditional methods such as the optimal Hohmann \citep{33,34} and Lambert solutions as well as the single impulse maneuver. In this sense, the main contributions of the current work can be summarized as follows: 
\begin{enumerate}[1-]
\item The SBGM can construct a smooth trajectory between any pair of coplanar oblique-elliptic orbits with any number of impulses; 
\item The SBGM reduces the number of required design variables in comparison with the existing methods such that it is computationally more efficient; 
\item The optimal solution of SBGM results in significantly lower control effort in comparison to the methods with similar number of design variables. 
\end{enumerate} 

The remaining parts of this paper are organized as follows: Basic geometric formulation as well as the shape-based impulsive maneuver algorithm are introduced in Section \ref{S:2}. Subsequently, Section \ref{S:3} is devoted to analyze circle-to-circle maneuvers for two- and three-impulse maneuvers. Applications of the proposed SBGM to the other types of orbital maneuvers are reported in Section \ref{S:4}. An extension of the SBGM for small variations is formulated in Section \ref{S:4p}. Finally, concluding remarks are presented in Section \ref{S:5}.

\section{Geometric Formulation and Requirements}
\label{S:2}

The purpose of this section is to develop the mathematical formulation needed to form an overall smooth trajectory constructed by joining a number of co-focal elliptical arcs. Since the satellite orbits are assumed in a two-body context, these elliptic arcs construct the satellite transfer trajectory after and before each impulse. 

\begin{definition}
A trajectory curve is called smooth, if it has $G^0$ and $G^1$ continuities, i.e., the curves intersect and share identical tangent direction at the intersection points \citep{35}.
\end{definition}

\begin{remark}
A smooth trajectory does not necessarily share a common center of curvature at the intersection points ($G^2$ continuity), unless the eccentricity of the connecting arcs is the same.
\end{remark}

The following assumptions are further stipulated for the intended trajectories: 

\begin{assumption}
\label{ass1}
The trajectory segments are elliptic, with the Earth center as a focus. 
\end{assumption}
\begin{assumption}
\label{ass2}
The trajectory segments are smooth at the intersection points. 
\end{assumption}

Considering an Earth centered inertial (ECI) reference coordinate system with the $x$ and $y$ axes as defined in Fig. \ref{fig:1}, an elliptic arc can be defined in depicted polar coordinate system as: 
\begin{equation}
\label{eq:1}
r=\frac{a(1-e^2)}{1+e\cos(\theta+\omega)}
\end{equation}
where $r$ is the local radius, $a$ denotes the semi-major axis, and $e$ represents the orbital eccentricity. As shown in Fig. \ref{fig:1}, the angle $\omega$ defines the rotation of the pre-apsis with respect to $x$, and $\theta+\omega$ represents  the orbital true anomaly. In this sense, a set of orbital elements $\boldsymbol p =[a\quad e\quad \theta\quad \omega]^T$ can be introduced as a complete description of the spacecraft position on a planar trajectory. In turn, orbital elements can be constructed just with the use of the radius vector $\boldsymbol r$, and its corresponding velocity vector $\boldsymbol v$ in the reference ECI coordinate system as follows \citep{36}: 
\begin{equation}
\label{eq:2}
e=\left\|\frac{rv^2}{\mu}\left[\widehat{\boldsymbol r}-(\widehat{\boldsymbol r}^T\widehat{\boldsymbol v})\widehat{\boldsymbol v}\right]-\widehat{\boldsymbol r}\right\|
\end{equation}
\begin{equation}
\label{eq:3}
a=\frac{r^2v^2}{\mu(1-e^2)}\|\widehat{\boldsymbol r}\times \widehat{\boldsymbol v}\|^2
\end{equation}
\begin{equation}
\label{eq:4}
\theta=\tan^{-1}\left(\frac{\widehat{r}_y}{\widehat{r}_x}\right)
\end{equation}
\begin{equation}
\label{eq:5}
\omega=\cos^{-1}\left\{\frac{1}{e}\left[\frac{a}{r}(1-e^2)-1\right]\right\}-\theta
\end{equation}
where $\|\cdot\|$ stands for the Euclidean norm, $r=\|\boldsymbol r\|$, $v=\|\boldsymbol v\|$, $\widehat{\boldsymbol r}=\boldsymbol r/r=\left[\widehat{r}_x\quad\widehat{r}_y\quad\widehat{r}_z\right]^T$, and $\widehat{\boldsymbol v}=\boldsymbol v/v$, respectively. The slope of a tangent line to the polar curve given in  Eq. \eqref{eq:1} is

\begin{equation}
\label{eq:6}
\frac{dr}{d\theta}=\frac{e\sin(\theta+\omega)}{1+e\cos(\theta+\omega)}r
\end{equation}

\begin{figure}[H]
\centering\includegraphics[width=0.5\linewidth]{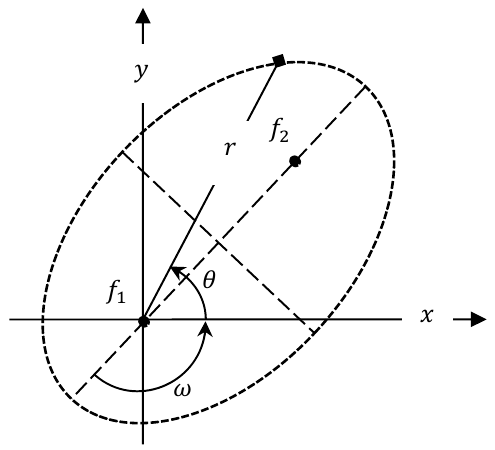}
\caption{Parameters of an oblique ellipse in polar coordinate system.}
\label{fig:1}
\end{figure}

In general we assume that a complete  orbital maneuver uses $N$ impulses to transfer from a point on initial orbit to a final destination point on the final orbit. The initial orbit is indexed by $i=1$ and the final orbit is denoted by $i=N+1$. Consequently, the sequential arcs are indexed from $i=2$ to $N$. In addition, it is assumed that the initial and the final orbital elements, i.e., $a_1$, $e_1$, $\omega_1$, and $\theta_1$ as well as $a_{N+1}$, $e_{N+1}$, $\omega_{N+1}$, and $\theta_{N+1}$ are known beforehand. 

Obviously, for an $N$-impulse maneuver, $N-1$ arcs need to be designed. At the intersection of any two sequential arcs, such as  $i$ and $i+1$,the following equations need to be held to satisfy Assumptions \ref{ass1} and \ref{ass2}:
\begin{equation}
\label{eq:7}
r_i=r_{i+1}\equiv r_{i(i+1)}
\end{equation}
\begin{equation}
\label{eq:8}
\theta_i=\theta_{i+1}\equiv \theta_{i(i+1)}
\end{equation}
\begin{equation}
\label{eq:9}
\left(\frac{dr}{d\theta}\right)_i=\left(\frac{dr}{d\theta}\right)_{i+1}
\end{equation}

Utilizing Eqs. \eqref{eq:7} to \eqref{eq:9}, the following relations should be satisfied  for each intersection $i$: 
\begin{equation}
\label{eq:10}
\begin{split}
f_{i1}=e_i\sin(\theta_{i(i+1)}+\omega_i)+e_ie_{i+1}\sin(\omega_i-\omega_{i+1})- \\
e_{i+1}\sin(\theta_{i(i+1)}+\omega_{i+1})=0
\end{split}
\end{equation}

\begin{equation}
\label{eq:11}
\begin{split}
f_{i2}=a_i(1-e_i^2)\left[1+e_{i+1}\cos(\theta_{i(i+1)}+\omega_{i+1})\right]- \\
a_{i+1}(1-e_{i+1}^2)\left[1+e_i\cos(\theta_{i(i+1)}+\omega_i)\right]=0
\end{split}
\end{equation}

As there are $N$ intersections,  consequently $2N$ equations will be developed, $\boldsymbol f=[f_{11}\quad\cdots
 \quad f_{N2}]^T$. While, the  unknown parameters are $a_i$, $e_i$, and $\omega_i$ for $i=2,...,N$, plus $\theta_{i(i+1)}$ for $i=2,...,N-1$, there would be $4N-5$ unknowns to be determined. Since, the number of unknowns are more than the number of equations for $N\geq3$, some additional feasible assumptions on adjustable parameters, $\boldsymbol\epsilon\subset\{\boldsymbol p_2,\cdots,\boldsymbol p_{N-1}\}$ are required. Assume that $M=2N-5$ parameters are adjustable, then the number of unknowns and the equations will be equal that makes a feasible solution. 

\begin{proposition}
\label{prop1}
Suppose Assumptions \ref{ass1} and \ref{ass2} hold. Then, for an $N$-impulse orbital maneuver, the $i$th orbit, $i=2,...,N$, cannot be circular unless the impulse is exerted at the apsis point (apogee or perigee) of the previous and the following orbits. 
\end{proposition}

\begin{proof}
Suppose that $e_2=0$. Therefore, from Eq. \eqref{eq:10}: 
\begin{equation}
\label{eq:12}
f_{11}=e_1\sin(\theta_{12}+\omega_1)=0
\end{equation}
\begin{equation}
\label{eq:13}
f_{21}=e_3\sin(\theta_{23}+\omega_3)=0
\end{equation}
$f_{11}=0$ is satisfied if $\theta_{12}+\omega_1=0$, $\theta_{12}+\omega_1=\pi$, or $e_1=0$. If $e_1=0$, Eq. \eqref{eq:11} leads to $a_1=a_2$ and consequently no impulses need to be applied, that is not acceptable. In addition, $\theta_{12}+\omega_1=0$ or $\pi$, means that the impulse should be exerted at the perigee or the apogee of the first orbit, respectively. Moreover, $f_{21}=0$ is satisfied if $\theta_{23}+\omega_3=0$, $\theta_{23}+\omega_3=\pi$, or $e_3=0$. If $e_3=0$, Eq. \eqref{eq:11} leads to $a_3=a_2$ that is similarly not acceptable. Finally, $\theta_{23}+\omega_3=0$ or $\pi$, means that the impulse should be exerted at the perigee or apogee of the third orbit, respectively. This principle can be extended to other adjacent orbits, i.e., $i-1$, $i$, and $i+1$ as well.
\end{proof}

\begin{remark}
In the numerical computations of the SBGM algorithm, singularities may occur. If the $i$th orbit, $i=2,...,N$, is circular, the determinant of the Jacobian matrix of $\boldsymbol f$, required for numerical solution utilizing classical methods like Newton, becomes zero and then singularity would arise. In this sense, the initial guess should not contain zero eccentricities in a Newton-based solution. However, using other computational methods like GA can rescue the solution process from singularities.
\end{remark}

\begin{remark}
\label{rem:1}
It is convenient for impulsive maneuvers to apply the first impulse at the perigee since the semi-major axis is more sensitive to the perigee velocity. In this case, for a two-impulse maneuver, without loss of generality one can assume $\theta_{12}+\omega_1=0$ and $\omega_1=\omega_2=0$. Consequently, according to Eq. \eqref{eq:10}, $f_{11}=0$ can be eliminated from the system of equations. In addition, Eq. \eqref{eq:11} for $i=1$ reduces to $a_2=a_1(1-e_1)/(1-e_2)$ that can be substituted in Eq. \eqref{eq:11} for $i=2$ in order to eliminate $a_2$ and finally arrive at the following equation: 
\begin{equation}
\label{eq:13p}
\frac{a_1}{a_3}\frac{1-e_1}{1-e_3}(1+e_2)[1+e_3\cos(\theta_{23}+\omega_3)]-(1+e_3)[1+e_2\cos(\theta_{23})]=0
\end{equation}
Eq. \eqref{eq:10} for $i=2$ reduces to:
\begin{equation}
\label{eq:13pp}
e_2\sin(\theta_{23})-e_2e_3\sin(\omega_3)-e_3\sin(\theta_{23}+\omega_3)=0
\end{equation}
In this sense, Eqs. \eqref{eq:13p} and \eqref{eq:13pp} construct a system of two equations which can be used to solve two scalar unknowns $e_2$ and $\theta_{23}$. If however, the final orbit ($i=3$) is circular (ellipse-to-circle maneuver), $e_3=0$, such that $a_3\geq a_1$, the solutions of Eqs. \eqref{eq:13p} and \eqref{eq:13pp} reduce to: 
\begin{equation}
\label{eq:13ppp}
\theta_{23}=\pi, \quad\quad e_2=\frac{a_3-a_1(1-e_1)}{a_3+a_1(1-e_1)}
\end{equation}
\end{remark}

The interesting feature of the proposed geometric method is reducing the number of variables need to be adjusted. Table \ref{table:1} shows the design (or optimization) variables required to construct an $N$-impulse orbital maneuver in the conventional methods and compares them against the proposed SBGM. 

\begin{table}[H]
\caption{The number of design variables for $N$-impulse coplanar maneuvers}
\label{table:1}
\centering
\begin{tabular}{p{2.6 cm} p{2.5 cm} p{1.7 cm} p{2.5 cm} p{1.3 cm}} 
 \hline\hline
 Initial Impulse Location & Final Impulse Location & Maneuver Time & No. of Design Variables & Related Ref. \\ 
 \hline
 F  & F  & F  & $3N-6$ & \citep{05} \\ 
 F  & F  & NF & $3N-5$ & \citep{07} \\
 F  & NF & F  & $3N-4$ & \citep{02} \\
 F  & NF & NF & $3N-3$ & \citep{07} \\
 NF & NF & F  & $3N-2$ & \citep{02} \\
 NF & NF & NF & $3N-1$ & \citep{04} \\
 \hline
 \multicolumn{3}{c}{Proposed SBGM} & $2N-5$ & [$*$] \\
 \hline\hline
 \multicolumn{5}{l}{\scriptsize{No.: Number, Ref.: Reference, F: Fixed, NF: Not Fixed, [$*$]: Current paper}} \\
\end{tabular}
\end{table}

Table \ref{table:1} shows that utilizing lower number of design variables (that is achieved by the use of smoothness property) for the orbital transfer is one of the key advantages of the proposed SBGM. The proposed algorithm needs two fixed initial and final conditions for initiation, but since the initial and final orbits are also candidate trajectories in the optimization process, SBGM can be regarded as a general orbit-to-orbit maneuver algorithm for problems in which the initial and the final impulse locations are not necessarily fixed. The SBGM algorithm is summarized in Algorithm \ref{al1}.

\begin{algorithm}[H]
\label{al1}
 \KwData{Initial orbit elements $\boldsymbol p_1$, final orbit elements $\boldsymbol p_N$, number of impulses $N$, and adjustable variables $\boldsymbol \epsilon$.}
 \KwResult{All orbital elements $\boldsymbol p_i$, $i=1,...,N$. }

 \For{$i=1$ to $N$}{
  \nl Make the system of nonlinear equations from Eqs. \eqref{eq:10} and \eqref{eq:11}: $\boldsymbol f_{2i-1:2i}\leftarrow [f_1(\boldsymbol p_1,\boldsymbol p_N,\boldsymbol \epsilon)\quad f_2(\boldsymbol p_1,\boldsymbol p_N,\boldsymbol \epsilon)]^T$
 }
 \nl Use an algorithm such as the relaxed Newton method \citep{37} to solve the system of equations made in previous step: $(\boldsymbol p_2,...,\boldsymbol p_{N-1})\leftarrow\text{Solution}\{\boldsymbol f(\boldsymbol p_2,\boldsymbol p_{N-1},\boldsymbol \epsilon)=[0]_{2N\times 1}\}$ 
 
 \Return{$\boldsymbol p_i$, $i=1,...,N$}
 \caption{SBGM Algorithm for tangent trajectory generation}
\end{algorithm}

Upon the returning of unknowns and  $\boldsymbol p_i$, $i=1,...,N$, the impulse vectors can be calculated. In this respect, the velocity before and after applying the $i$th impulse is denoted by $v_i^-$ and $v_i^+$, respectively. Therefore, the impulse magnitude can be calculated as:
\begin{equation}
\label{eq:14}
\Delta v_i=v_i^+-v_i^-
\end{equation}
where
$$v_i^-=\sqrt{\mu\left(\frac{2}{r_i}-\frac{1}{a_i}\right)} \quad\quad v_i^+=\sqrt{\mu\left(\frac{2}{r_i}-\frac{1}{a_{i+1}}\right)}$$
and $a_i$ is the semi-major axis of the $i$th orbit. 

\section{Maneuver between Circular Orbits}
\label{S:3}

In this section two propositions are proven in order to find the links between the proposed SBGM and the well-known Hohmann transfer as well as its three-impulse counterpart, i. e., the bi-elliptic maneuver. First, consider the following lemma that is another interpretation of Proposition \ref{prop1} indeed.  

\begin{lemma}
\label{lem1}
According to the Assumptions \ref{ass1} and \ref{ass2}, if the previous (or the following) arc is circular, $e_{i-1}=0$ (or $e_{i+1}=0$), any required impulse should be exerted at an apsis point of the $i$th orbit.  
\end{lemma}

\begin{proof}
Case 1: The previous orbit is circular, $e_{i-1}=0$. Thus, the corresponding argument of perigee can be eliminated, $\omega_{i-1}=0$. In this regard, according to Eq. \eqref{eq:10}:

\begin{equation}
\label{eq:15}
e_i\sin(\theta_{(i-1)i}+\omega_i)=0
\end{equation}
Therefore, three situations can be considered:
\begin{enumerate}[(\text{1-}1)]
\item $\theta_{(i-1)i}+\omega_i=0$;
\item $\theta_{(i-1)i}+\omega_i=\pi$;
\item $e_i=0$.
\end{enumerate}

According to Eq. \eqref{eq:11}, the conditions (1-1) to (1-3) lead to the following equations:
\begin{equation}
\label{eq:16}
a_i=a_{i-1}\frac{1}{1-e_i}
\end{equation}
\begin{equation}
\label{eq:17}
a_i=a_{i-1}\frac{1}{1+e_i}
\end{equation}
\begin{equation}
\label{eq:18}
a_i=a_{i-1}
\end{equation}

Case 2: The following orbit is circular, $e_{i+1}=0$. Similarly, three conditions can be considered in this case: 

\begin{enumerate}[(\text{2-}1)]
\item $\theta_{i(i+1)}+\omega_i=0$;
\item $\theta_{i(i+1)}+\omega_i=\pi$;
\item $e_i=0$.
\end{enumerate}

According to Eq. \eqref{eq:11}, the conditions (2-1) to (2-3) lead to the following equations:
\begin{equation}
\label{eq:19}
a_i=a_{i+1}\frac{1}{1-e_i}
\end{equation}
\begin{equation}
\label{eq:20}
a_i=a_{i+1}\frac{1}{1+e_i}
\end{equation}
\begin{equation}
\label{eq:21}
a_i=a_{i+1}
\end{equation}

Conditions (1-1), (1-2), (2-1), and (2-2) (Eqs. \eqref{eq:16}, \eqref{eq:17}, \eqref{eq:19}, and \eqref{eq:20}) describe impulse positions corresponding to the perigee and apogee, respectively. Conditions (1-3) and (2-3) mean that no impulses should be applied and consequently the $N$-impulse maneuver reduces to an $(N-1)$-impulse maneuver.
\end{proof}

Eccentricity of the $i$th orbit, depends on the following subsequent trajectory. Solution of Eqs. \eqref{eq:16}, \eqref{eq:17}, \eqref{eq:19}, or \eqref{eq:20} requires $e_i$ to be known.

\begin{proposition}
\label{prop2}
Under Assumptions \ref{ass1} and \ref{ass2}, for a two-impulse maneuver between circular orbits, the only solution is the well-known Hohmann maneuver. 
\end{proposition}

\begin{proof}
For a two-impulse maneuver between orbits $i=1$ and $i=3$, a single orbit, $i=2$, is required to allow the transfer mission. Thus, from Lemma \ref{lem1}, conditions (1-1) or (1-2) should be observed beside the conditions (2-1) or (2-2). Conditions (1-1) and (2-1) or similarly conditions (1-2) and (2-2) are inconsistent and lead to $a_1=a_3$. Conditions (1-1) and (2-2) yields
\begin{equation}
\label{eq:22}
e_2=\frac{a_3-a_1}{a_3+a_1}
\end{equation}
that guarantees $0<e_2<1$. Condition (1-2) beside (2-1) yield a negative eccentricity that is invalid. Considering the Hohmann transfer orbit characteristics, Eq. \eqref{eq:22} thus introduces the common Hohmann \citep{01} solution. 
\end{proof}

\begin{proposition}
\label{prop3}
Considering Assumptions \ref{ass1} and \ref{ass2} for a three-impulse maneuver between circular orbits, if the location of two impulses lie on a straight line passing through the focus, the solution is the bi-elliptic Hohmann maneuver.
\end{proposition}

\begin{proof}
According to Lemma \ref{lem1}, the first and the third impulses should be applied at the perigee (or the apogee) of the second and the third orbits, respectively. Therefore, since the perigee and the apogee of an ellipse lie on a straight line passing through the focus, the only solution that satisfies both Assumptions \ref{ass1} and \ref{ass2} will require the impulses to be applied only at two true anomalies of either $0$ or $\pi$. In this sense, the impulses should be exerted on the perigee or apogee of the second and the third orbits, and consequently a bi-elliptic maneuver \citep{36} develops. 
\end{proof}

\begin{remark}
For a general three-impulse maneuver between circular orbits, the impulse locations should not necessarily lie on the same apse line to satisfy Assumptions \ref{ass1} and \ref{ass2}.
\end{remark}
Consequently, the well-known Hohmann transfer and the bi-elliptic method are special cases of the general multiple impulse trajectory problem proposed here. 
 
\section{Optimal Solutions}
\label{S:4}
The proposed geometric method, SBGM can be improved via an optimization process. A control effort (CE) cost function is taken to minimize the sum of impulse norms: 
\begin{equation}
\label{eq:23}
J_c=\sum_{j=1}^{N}|\Delta v_j|
\end{equation}
In addition, the maximum impulse (MI) cost function is taken to minimize the maximum impulse norm:
\begin{equation}
\label{eq:24}
J_m=\max_{j=1,...,N}\left(|\Delta v_j|\right)
\end{equation}
Two case studies are analyzed here for which the above mentioned cost functions are utilized ($J_c$ and $J_m$): The first case investigates an orbital maneuver to circularize a low Earth orbit (LEO), and the second one is the problem of trajectory design to transfer from a LEO to a Molniya orbit. Both problems are solved utilizing two and three-impulse maneuvers.

Without loss of generality, the angles $\omega$ and $\theta$ can be assumed bounded in the $[0, 2\pi]$ domain. In this sense, selection of these angles ($\omega$ and/or $\theta$) as the adjustable parameters is justified within the bounded search region. For the case of $N=3$, the second argument of perigee ($\omega_2$) is considered as the adjustable parameter. It is the only required parameter to determine the optimum solution. Upon selection of the adjustable parameters, a system of nine nonlinear equations will be constructed from Eqs. \eqref{eq:10} and \eqref{eq:11} for $i=1$, $2$, and $3$. This system of nonlinear equations is solved using a relaxed Newton algorithm. 

The optimal solution of the proposed two- and three-impulse SBGM  is compared with the optimized two-impulse Lambert solution for the same cost functions. The Lambert problem can be solved knowing the initial position ($\boldsymbol r_1$), final position ($\boldsymbol r_f$) as well as the total maneuver time ($t_f$). However, sometimes the final orbit is known but the final position is not fixed, i.e., the $\boldsymbol r_f=[r_f\cos\theta_f\quad r_f\sin\theta_f]^T$ is a function of the final polar angle $\theta_f$ as $r_f=a_f(1-e_f^2)/\left[1+e_f\cos⁡(\theta_f+\omega_f)\right]$. Accordingly, the Lambert solution can be optimized using two different optimization scenarios:
\begin{enumerate}[(a)]
\item Parameter $t_f$ is the only optimization variable; 
\item Parameters $t_f$ and $\theta_f$ are the optimization variables.
\end{enumerate}

In the first Lambert scenario (a), the transfer time is taken as the optimization variable to construct an orbital arc that joins the initial and final points. The second Lambert scenario solution (b) is determined using two optimization variables; transfer time and the polar angle at the arrival. Further, the second Lambert solution (b) allows the spacecraft to remain in its initial orbit before applying the first impulse, and this could lead to trajectories with lower cost. The Lambert solution is optimized by a simple search space algorithm. It is also notable that Subsection \ref{sub:4-3} is devoted to discuss results obtained in two case studies. 

\subsection{Case Study 1: Transfer of an Eccentric LEO to a Circular LEO}
\label{sub:4-1}

In this study, it is tried to transfer from a LEO orbit with high eccentricity to a circular LEO orbit at the same orbital plane. The initial and final orbital elements are taken as follows:  

\begin{center}
\begin{tabular}{lllll}
$a_1=13756\text{ km}$ & $e_1=0.5$ & $\omega_1=10^\circ$ & $\theta_{12}=270^\circ$ \\
$a_4=13756\text{ km}$ & $e_4=0$ & $\omega_4=60^\circ$ & $\theta_{34}=30^\circ$ \\
\end{tabular}
\end{center}

There is a solution for each value of $\omega_2$ unless a singularity occurs. The obtained solutions as a function of $\omega_2$ lead to the results depicted in Fig. \ref{fig:2} for the impulse magnitudes, CE and MI penalties, as well as the maneuver time, respectively. It can be easily deduced from these results that two solutions are exist in which one of the impulse values are zero that can be selected as the two-impulse solutions (the two-impulse solutions can be solely determined by solving a system of equations without any adjustable variable). 

\begin{figure}[H]
\centering\includegraphics[width=0.8\linewidth]{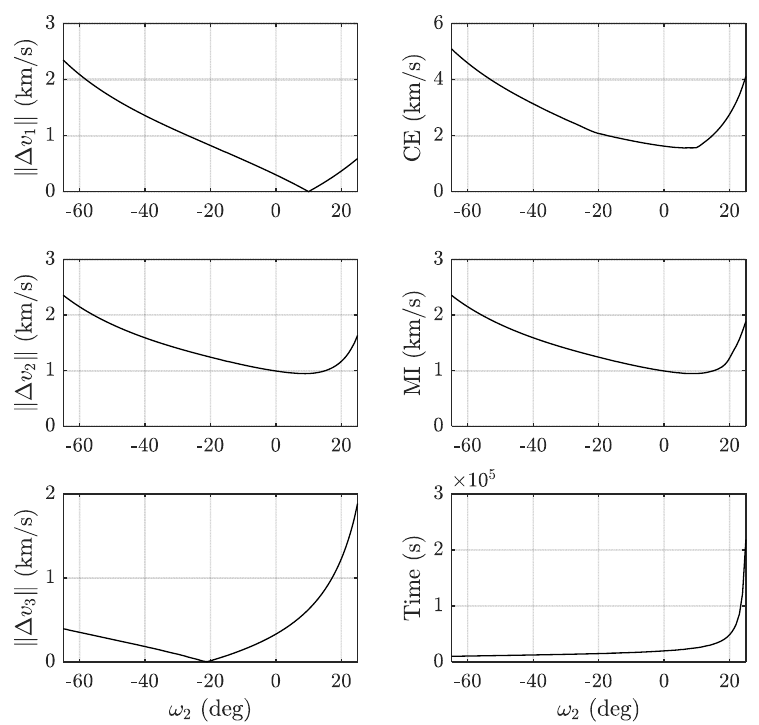}
\caption{Maneuver impulse, cost function, and time against $\omega_2$ for SBGM (case study 1).}
\label{fig:2}
\end{figure}

Optimal solutions for the SBGM is also obtained with the two cost functions of CE and MI, where their results are compared with the pertinent optimal two-impulse Lambert solution of scenarios (a) and (b). The obtained results are summarized in Table \ref{table:2}. In addition, Figs. \ref{fig:3} and \ref{fig:4} show the trajectories associated with the optimal solutions of CE and MI, respectively. The two-impulse SBGM solutions (corresponding to the points that $|\Delta v_1|=0$ and $|\Delta v_3|=0$ in Fig. \ref{fig:2}) are shown in Fig. \ref{fig:5} in which one of them is the optimal solution.

The orbits of case study 1 are intersecting and among the traditional methods, the single impulse maneuver can be applied for this scenario. Moreover, the two-impulse maneuver with the first impulse applied at the perigee (discussed in Remark \ref{rem:1}) can be considered and calculated analytically since the target orbit is circular. The details are presented in Table \ref{table:2}, in which the single impulse maneuver has two distinct solutions and the above-mentioned two-impulse maneuver has a unique solution that are denoted by ``1-impulse'' and ``2-impulse (Remark \ref{rem:1})'', respectively.

\begin{table}[H]
\caption{Summary of the optimal results for case study 1.}
\label{table:2}
\centering
\begin{tabular}{p{2.4 cm} p{0.7 cm} p{3.1 cm} p{2.3 cm} p{1.5 cm} p{0.5cm}} 
 \hline\hline
 Method & Cost & Opt. Sol. (km/s) & Trans. T. (s) & Opt.Var. & Sm. \\ 
 \hline
 \multirow{2}{*}{1-Impulse} &  & \multirow{2}{*}{$J_c=2.6305$} & $t_f=2631$ & \multirow{2}{*}{$-$} & No \\ 
  &  &  & $t_f=3463$ &  & No \\ 
 \hdashline
 2-Impulse & CE & $J_c^*=4.4539$ & $t_f=3750$ & \multirow{2}{*}{$t_f$} & No \\ 
 Lambert (a) & MI & $J_m^*=2.2989$ & $t_f=3184$ &  & No \\ 
 \hdashline
 2-Impulse & CE & $J_c^*=1.4677$ & $t_f=11640$ & \multirow{2}{*}{$t_f$, $\theta_f$} & No \\ 
 Lambert (b) & MI & $J_m^*=0.7831$ & $t_f=12090$ &  & No \\ 
 \hdashline
 2-Impulse &  & $J_c=1.5210$ & \multirow{2}{*}{$t_f=2315$} & \multirow{2}{*}{$-$} & \multirow{2}{*}{Yes} \\ 
 (Remark \ref{rem:1}) & & $J_m=0.9878$ & &  &  \\ 
 \hdashline
 2-Impulse &  & $J_c=1.5746$ & \multirow{2}{*}{$t_f=25415$} & \multirow{2}{*}{$-$} & \multirow{2}{*}{Yes} \\ 
 Shape-Based &  & $J_m=0.9487$ &  &  &  \\ 
 \hdashline
 3-Impulse & CE & $J_c^*=1.5746$ & $t_f=24581$ & \multirow{2}{*}{$\omega_2$} & Yes \\ 
 Shape-Based & MI & $J_m^*=0.9471$ & $t_f=23156$ &  & Yes \\ 
 \hline\hline
 \multicolumn{6}{@{} l}{\scriptsize{Opt.: Optimal, Sol.: Solution, Trans.: Transfer, T.: Time, Var.: Variable(s), Sm.: Smoothness}}
\end{tabular}
\end{table}

\begin{figure}[H]
\centering\includegraphics[width=0.7\linewidth]{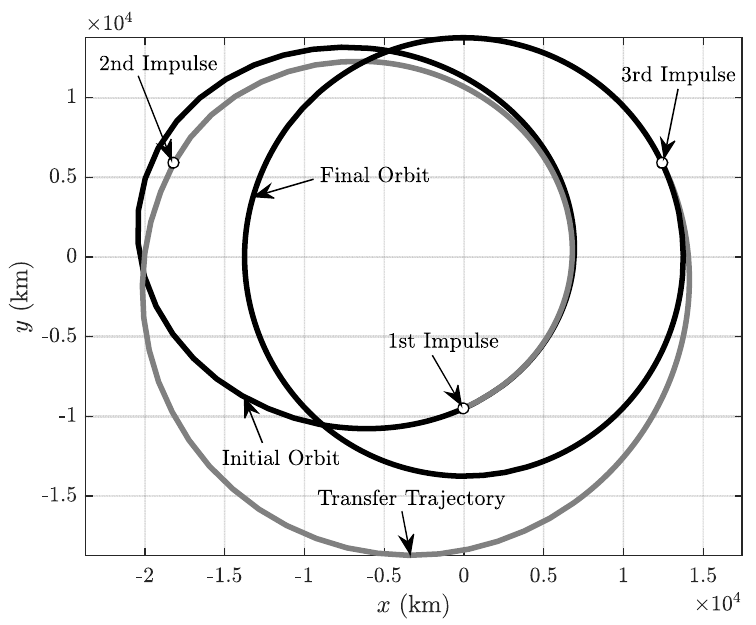}
\caption{Optimal SBGM trajectory with CE cost function for case study 1.}
\label{fig:3}
\end{figure}

\begin{figure}[H]
\centering\includegraphics[width=0.7\linewidth]{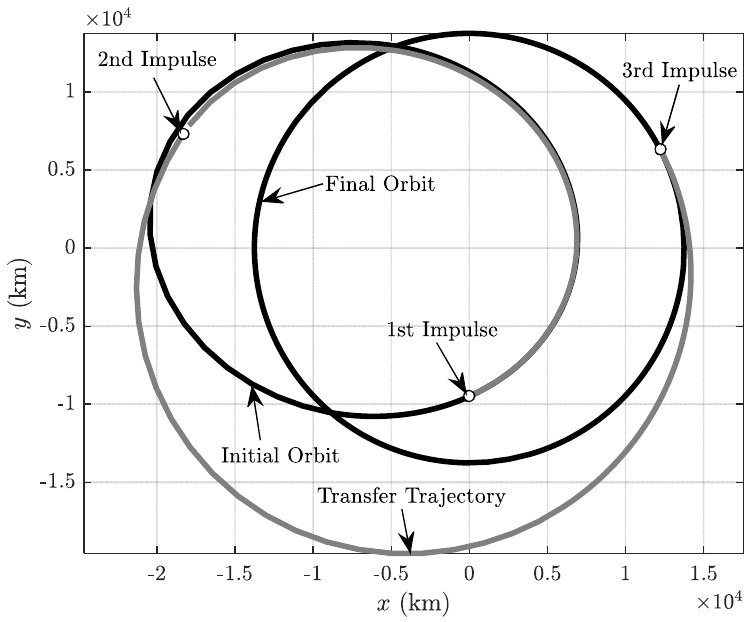}
\caption{Optimal SBGM trajectory with MI cost function for case study 1.}
\label{fig:4}
\end{figure}

\begin{figure}[H]
\centering\includegraphics[width=0.7\linewidth]{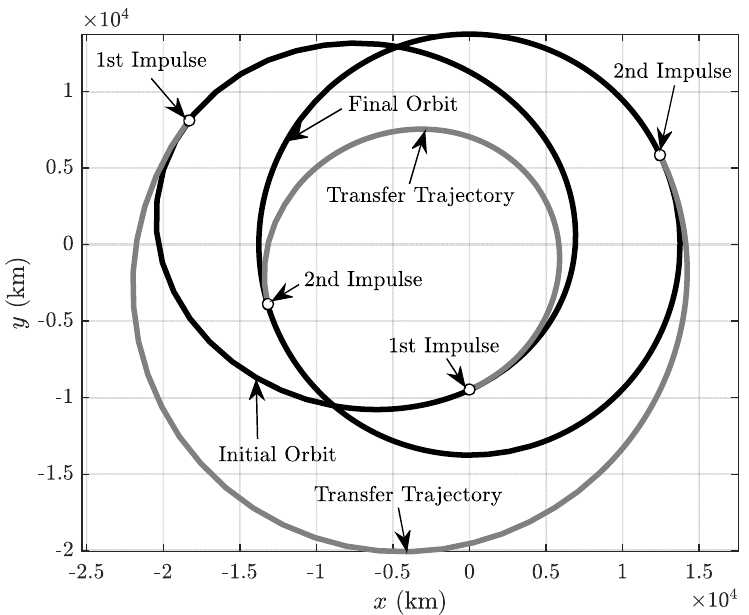}
\caption{Two-impulse SBGM trajectories for case study 1.}
\label{fig:5}
\end{figure}

\subsection{Case Study 2: A LEO to Molniya Transfer}
\label{sub:4-2}

The problem is to transfer a spacecraft from an eccentric LEO to a Molniya orbit. The initial and the final orbital elements are assumed as: 

\begin{center}
\begin{tabular}{lllll}
$a_1=6644.4\text{ km}$ & $e_1=0.01$ & $\omega_1=60^\circ$ & $\theta_{12}=45^\circ$ \\
$a_4=26562\text{ km}$ & $e_4=0.74105$ & $\omega_4=30^\circ$ & $\theta_{34}=15^\circ$ \\
\end{tabular}
\end{center}

The pertinent optimal Lambert solutions as a function of $\omega_2$ lead to the results provided in  Fig. \ref{fig:6} for the impulse values, CE, MI, and the maneuver time, respectively. 

\begin{figure}[H]
\centering\includegraphics[width=0.8\linewidth]{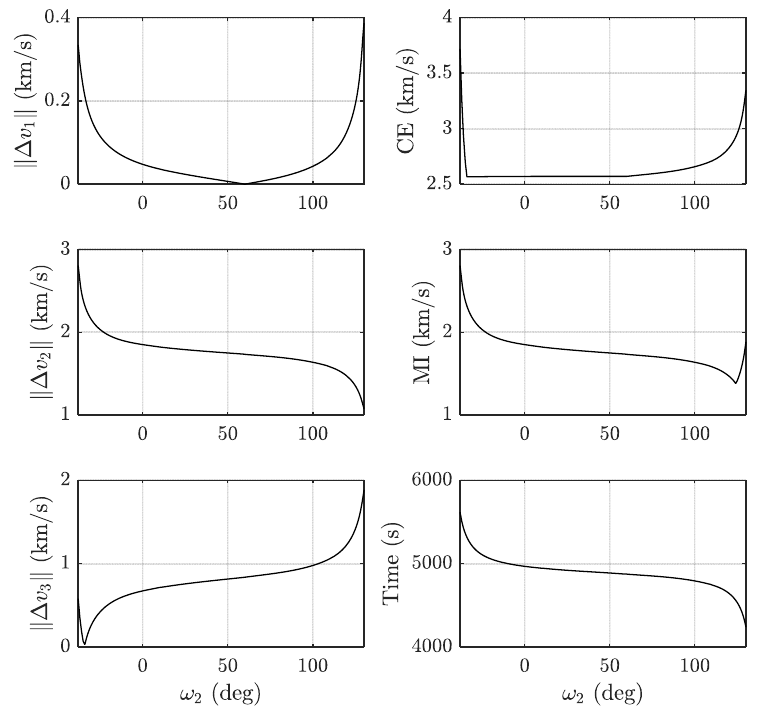}
\caption{Maneuver impulses, cost functions, and time against $\omega_2$ for SBGM (case study 2).}
\label{fig:6}
\end{figure}

The optimal solutions of SBGM are also determined using the previously mentioned cost functions CE and MI, respectively. Again, SBGM results are compared with the optimal two-impulse Lambert solutions scenarios in Table \ref{table:3}. Figs. \ref{fig:7} and \ref{fig:8} also show trajectories associated with the CE and MI cost functions, respectively. 

\begin{table}[H]
\caption{Summarized optimal results for case study 2}
\label{table:3}
\centering
\begin{tabular}{p{2.4 cm} p{0.7 cm} p{3.1 cm} p{2.3 cm} p{1.5 cm} p{0.5cm}} 
 \hline\hline
 Method & Cost & Opt. Sol. (km/s) & Trans. T. (s) & Opt.Var. & Sm. \\ 
 \hline
 2-Impulse & CE & $J_c^*=5.1176$ & $t_f=2894$ & \multirow{2}{*}{$t_f$} & No \\ 
 Lambert (a) & MI & $J_m^*=7.9455$ & $t_f=3570$ &  & No \\ 
 \hdashline
 2-Impulse & CE & $J_c^*=1.3344$ & $t_f=826$ & \multirow{2}{*}{$t_f$, $\theta_f$} & No \\ 
 Lambert (b) & MI & $J_m^*=2.5604$ & $t_f=764$ &  & No \\ 
 \hdashline
 2-Impulse &  & $J_c=2.3263$ & \multirow{2}{*}{$t_f=5009$} & \multirow{2}{*}{$-$} & \multirow{2}{*}{Yes} \\ 
 Shape-Based &  & $J_m=2.5659$ &  &  &  \\ 
 \hdashline
 3-Impulse & CE & $J_c^*=1.3815$ & $t_f=4560$ & \multirow{2}{*}{$\omega_2$} & Yes \\ 
 Shape-Based & MI & $J_m^*=2.5659$ & $t_f=5009$ &  & Yes \\ 
 \hline\hline
 \multicolumn{6}{@{} l}{\scriptsize{Opt.: Optimal, Sol.: Solution, Trans.: Transfer, T.: Time, Var.: Variable(s), Sm.: Smoothness}}
\end{tabular}
\end{table}

\begin{figure}[H]
\centering\includegraphics[width=0.6\linewidth]{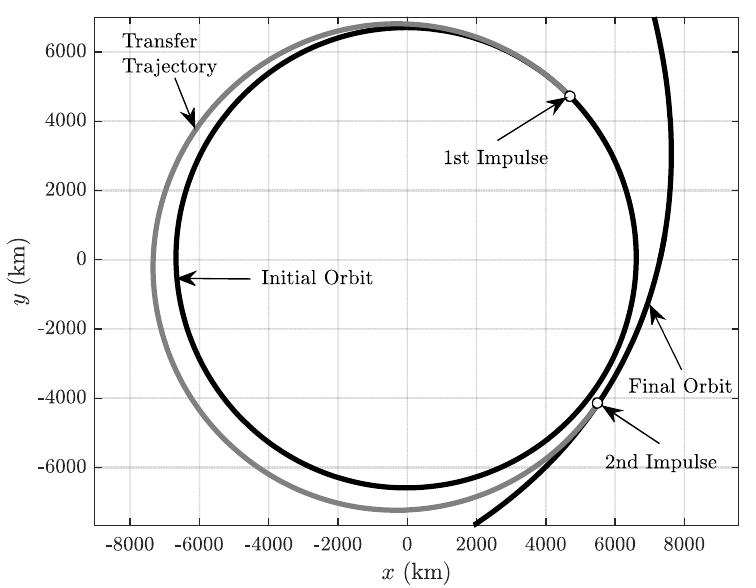}
\caption{Optimal SBGM trajectory with CE cost function for case study 2.}
\label{fig:7}
\end{figure}

\begin{figure}[H]
\centering\includegraphics[width=0.6\linewidth]{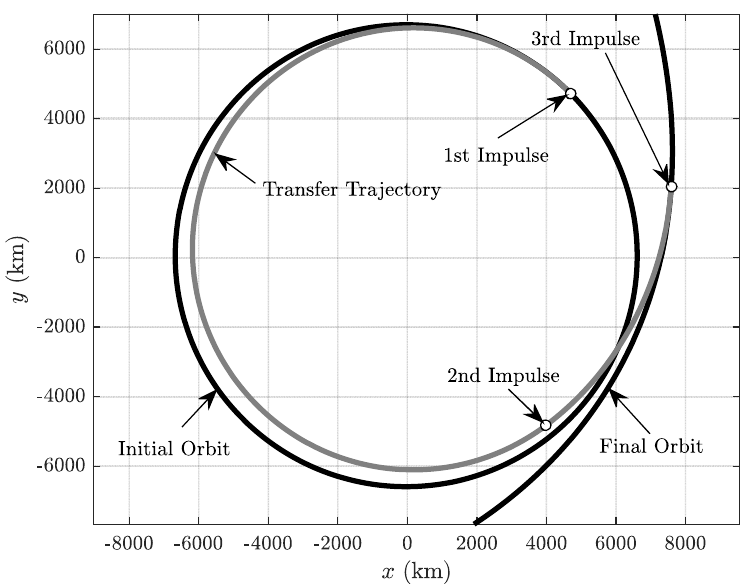}
\caption{Optimal SBGM trajectory with MI cost function for case study 2.}
\label{fig:8}
\end{figure}

\subsection{Discussion}
\label{sub:4-3}

This section is devoted to discuss the obtained results in Subsections \ref{sub:4-1} and \ref{sub:4-2} where the SBGM as well as the optimized Lambert solution under two described scenarios have been utilized to study the coplanar transfers. The SBGM is solved for two and three-impulse maneuvers, where the optimization cost functions are based on two criteria of cumulative control efforts as well as the maximum impulse level. It can be assumed that the computational efficiency is proportional to the number of design variables $n$ ($\text{runtime}\propto n$). The number of function evaluations for solving step 2 in Algorithm \ref{al1} in the numerical results of Subsections \ref{sub:4-1} and \ref{sub:4-2} are $(12+1)\times2000=26\times10^3$. Twelve evaluations for the computation of the Jacobian matrix plus a single evaluation for the function itself multiplied by $2\times10^3$ iterations. In comparison with the Lambert's method, the number of function evaluations of the SBGM is larger owing to the complexity of the associated system of nonlinear equations.

Comparison of the obtained results provided in Tables \ref{table:2} and \ref{table:3} conveys that: 

\begin{enumerate}[1-]
\item The two-impulse Lambert solution (b) leads to the best (lowest) costs, while the two-impulse Lambert solution (a) leads to the worst (highest) costs.
\item The two- and three-impulse SBGM results are close to the two-impulse Lambert (b) solution.  
\item The two-impulse SBGM does not need any optimization variable, in comparison to the Lambert (b) solution requiring at least two optimization variables. 
\item As can be expected, the three-impulse SBGM has better (case study 1) or at least equal (case study 2) costs in comparison with the two-impulse SBGM, since the two-impulse is a special solution of the three-impulse case. 
\item The three-impulse SBGM solution has just a single optimization variable, then it is more computationally efficient than the two-impulse Lambert (b) solution that requires two optimization variables.  
\item The two-impulse solution with the first impulse applied at the perigee (Remark \ref{rem:1}) has lower cost in comparison to the two- and three-impulse SBGM. The reason is that the two- and three-impulse SBGM are constrained at the arrival and the destination points. Moreover, the three-impulse SBGM (in contrast with the two-impulse solutions) has adjustable parameters and can be used in a trade-off by a designer. 
\end{enumerate}

In this sense, the proposed SBGM is advantageous over existing impulsive strategies at least from two important points of view including a suitable reduction in the required design variables for optimization plus rendering solutions with considerably lower impulsive energy. 

\section{Variational Approach}
\label{S:4p}

This section discusses shaping a transfer orbit under the constraint of small variations between two successive arbitrary orbital parameters, i.e., $a_{i+1}=a_i+\delta a$, $e_{i+1}=e_i+\delta e$, $\omega_{i+1}=\omega_i+\delta\omega$ where $\delta a$, $\delta e$, and $\delta\omega$ are small enough such that their second order Taylor expansion terms can be ignored. In this regard, Eqs. \eqref{eq:10} and \eqref{eq:11} can be rewritten as follows (subscript $i$ is removed for simplicity): 

\begin{equation}
\label{eq:25}
e\left[e+\cos(\theta+\omega)\right]\delta\omega+\sin(\theta+\omega)\delta e=0
\end{equation}

\begin{equation}
\label{eq:26}
e\sin(\theta+\omega)\delta\omega-\frac{(1+e^2)\cos(\theta+\omega)+2e}{1-e^2}\delta e+\frac{1+e\cos(\theta+\omega)}{a}\delta a=0
\end{equation}

As mentioned before, Eqs. \eqref{eq:25} and \eqref{eq:26} can be used for impulsive orbital maneuvers with small changes in orbital elements. Substituting difference values of $\delta a$, $\delta e$, and $\delta\omega$ with differentials of $da/d\theta$, $de/d\theta$, and $d\omega/d\theta$ leads to a set of equations that should be considered as trajectory constraints in a tangential continuous-thrust maneuver. 

As a simple intuitive example, for an ellipse-to-circle maneuver in which the only objective is reducing the orbital eccentricity, an arbitrary decreasing pattern for $e$ dependent on $\theta$ can be considered. For instance, assume that $e$ decreases as $e=e_0\exp(-\alpha\theta)$ and $de/d\theta=-\alpha e_0\exp(-\alpha\theta)$ for any $\alpha>0$ where $e_0$ is the eccentricity of the initial orbit. Using Eqs. \eqref{eq:25} and \eqref{eq:26} the values of $da/d\theta$ and $d\omega/d\theta$ can be obtained as a function of $\theta$. Fig. \ref{fig:9} shows the pertinent case study where the eccentricity of an elliptic orbit converged exponentially to zero, regardless of the final values of $a$ and $\omega$ after ten periods of implementing continuous changes in orbital elements. 

\begin{figure}[H]
\centering\includegraphics[width=0.9\linewidth]{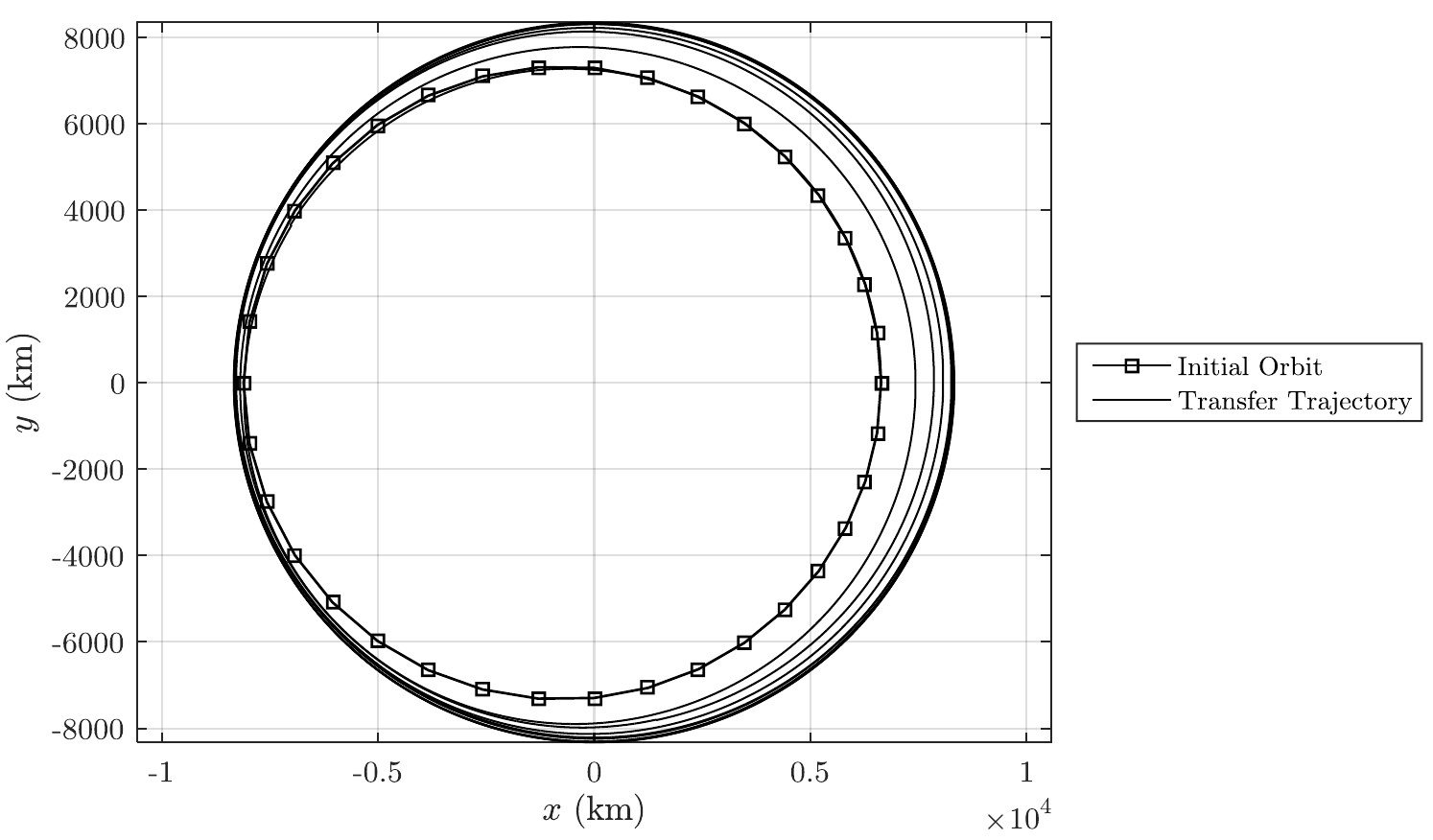}
\caption{A continuous-thrust maneuver: The eccentricity converges exponentially from $0.1$ to zero, $e(\theta)=0.1\exp(-\alpha\theta)$, $\alpha=0.1$. Perigee altitude of the initial orbit is $1000\text{ km}$.}
\label{fig:9}
\end{figure}

\section{Conclusions}
\label{S:5}

A new shape-based multiple-impulse maneuver strategy is proposed for coplanar orbital transfers. It is demonstrated that the well-known two-impulse Hohmann and its counterpart three-impulse bi-elliptic alternative are special cases of the proposed shape-based geometric method (SBGM). In this respect, the proposed shape-based approach is utilized for optimal two and three-impulse transfers between arbitrary fixed (spatial) points in different coplanar oblique elliptical orbits. In addition, optimized two-impulse Lambert solutions are also investigated for comparison purposes with those of the proposed SBGM. The comparative results show that SBGM dramatically reduces the number of required design parameters of the developed optimization problem, while at the same time more efficient transfer trajectories are produced in terms of energy requirements. In other words, the resulting advantages outweigh the extra mathematical effort needed to solve a system of nonlinear equations that emanates out of the proposed scheme. In this sense the proposed SBGM could be a viable alternative for coplanar orbital maneuvers and rendezvous missions with any arbitrary number of impulses. Extension of the proposed method to the small variation cases shows that a continuous tangential maneuver can be achieved by considering a set of trajectory constraints which are the necessary and sufficient conditions to satisfy the smoothness of the transfer trajectory. 

%\section*{References}

\bibliography{mybibfile}

\end{document}